\documentclass[a4paper, 12pt]{amsart}
\usepackage{amsmath,calc}
\usepackage{amssymb}
\usepackage{amsthm}
\usepackage[mathscr]{euscript}
\usepackage{subfiles}
\usepackage{blindtext}
\usepackage{bibentry}
\usepackage{cases}
\usepackage{bm}

\DeclareMathOperator{\Ima}{Im}

\DeclareMathOperator{\Span}{span}

\DeclareMathOperator{\Sign}{sgn}

\theoremstyle{plain}
\newtheorem{theorem}{Theorem}[section]

\newtheorem{lemma}[theorem]{Lemma}
\newtheorem{corollary}[theorem]{Corollary}
\theoremstyle{remark}
\newtheorem{remark}[theorem]{Remark}
\theoremstyle{definition}
\newtheorem{definition}[theorem]{Definition}

\begin{document}

\title{Flat approximations of hypersurfaces along curves}

\author{Irina Markina}
\address{Department of Mathematics\\
University of Bergen\\
5020 Bergen\\
Norway}
\email{irina.markina@uib.no}

\author{Matteo Raffaelli}
\address{DTU Compute\\
Technical University of Denmark\\
2800 Kongens Lyngby\\
Denmark}
\email{matraf@dtu.dk}

\date{\today}
\begin{abstract}
Given a smooth curve $\gamma$ in some $m$-dimensional surface $M$ in $\mathbb{R}^{m+1}$, we study existence and uniqueness of a flat surface $H$ having the same field of normal vectors as $M$ along $\gamma$, which we call a flat approximation of $M$ along $\gamma$. In particular, the well-known characterisation of flat surfaces as torses (ruled surfaces with tangent plane stable along the rulings) allows us to give an explicit parametric construction of such approximation.
\end{abstract}
\maketitle

\section{Introduction and Main Result}
\noindent Developable, or flat, hypersurfaces in $\mathbb{R}^{m+1}$, where $m\geq 2$, are classical objects in Riemannian geometry. They are characterised by being foliated by open subsets of $(m-1)$-dimensional planes, called rulings, along which the tangent space remains stable \cite[Theorem~1]{ushakov1999}. Here we are concerned with the problem of existence and uniqueness---as well as with the explicit construction---of flat approximations of hypersurfaces along curves. Let $M^{m}$ be a (possibly curved) Euclidean hypersurface and $\gamma$ a curve in $M^{m}$. A hypersurface $H$ is called an \textit{approximation of} $M^{m}$ \textit{along} $\gamma$ if the two manifolds have common tangent space at every point of $\gamma$.

In dimension $2$, the question of existence has been settled for a long time. A constructive proof, under suitable assumptions, is already present in Do Carmo's textbook \cite[p.~195--197]{docarmo1976}. It turns out the existence of a flat approximation of $M^{2}$ along $\gamma$ implies the existence of a rolling, in Nomizu's sense, of $M^{2}$ on the tangent space $T_{\gamma(0)}M^{2}$ along the given curve -- see \cite{nomizu1978} and \cite{raffaelli2016}. More recently, Izumiya and Otani have shown uniqueness \cite[Corollary~6.2]{izumiya2015}. 

In this paper, we extend the result in \cite{docarmo1976} to any curve in $M^{m}$. More precisely, we shall present a constructive proof of the following
\begin{theorem} \label{mainresult}
Let $\gamma \colon I \to M^{m}$ be a smooth curve in a hypersurface $M^{m}$ in $\mathbb{R}^{m+1}$. If the curve is never parallel to an asymptotic direction of $M^{m}$, then there exists a flat approximation $H$ of $M^{m}$ along $\gamma$. Such hypersurface is unique in the following sense: if $H_{1}$ and $H_{2}$ are two flat approximations of $M^{m}$ along $\gamma$, then they agree on an open set containing $\gamma(I)$.
\end{theorem}

The strategy to prove this result involves looking for $(m-1)$-tuples of linearly independent vector fields $(X_{1},\dotsc, X_{m-1})$ along $\gamma$ satisfying $\dot{\gamma}(t) \notin \Span(X_{j}(t))_{j\,=\,1}^{m-1}$ for all $t$ and having zero \emph{normal} derivative (normal projection of Euclidean covariant derivative). Indeed, such conditions guarantee the image of the map $\gamma + \Span(X_{j})_{j\,=\,1}^{m-1}$ be a flat hypersurface of $\mathbb{R}^{m+1}$ in a neighbourhood of $\gamma$. The main difficulty resides in getting around the many-to-one correspondence between tuples of vector fields and rank-$(m-1)$ distributions along $\gamma$.

It is worth pointing out that the solution depends on the original hypersurface $M^{m}$ only through its distribution of tangent planes along $\gamma$. Thus, when $m=2$, our problem is nothing but the classical Bj\"{o}rling's problem---to find all minimal surfaces passing through a given curve with prescribed tangent planes---addressed to a different class of surfaces. In this respect, the present work joins several other recent studies aimed at solving Bj\"{o}rling-type questions, see \cite{brander2018,brander2013} and references therein.

The paper is organised as follows. The next two sections present some preliminaries, mostly for the sake of introducing relevant notation and terminology. In Section $4$ we derive a simple condition for discerning when a parametrised ruled hypersurface has a flat metric. Such condition is then used in Section $5$ to prove the main theorem. Finally, in Section $6$ we give some general remarks about the construction of the approximation.

\section{Vector Cross Products}
\noindent Let $V$ be an $n$-dimensional, real vector space equipped with a positive definite inner product $\langle\cdot{,}\cdot\rangle$. In the following, $V^{k}$ will indicate the $k$-th Cartesian power of $V$, and $L^{k}(V)$ the set of all multilinear maps from $V^{k}$ to $V$. Note that, under pointwise addition and scalar multiplication, $L^{k}(V)$ is a \emph{finite dimensional} vector space, in that it is naturally isomorphic to the space $T^{(1,k)}(V)$ of tensors on $V$ of type $(1,k)$ -- see for example \cite[Lemma~2.1]{lee1997}. Thus, $\dim L^{k}(V) = n^{k+1}$.

A $k$-\textit{fold vector cross product on} $V$, $1 \leq k \leq n$, is an element of $L^{k}(V)$---i.e., a multilinear map $ X \colon V^{k} \to V$---satisfying the following two axioms:
\begin{align*}
&\langle X(v_{1}, \dotsc, v_{k}), v_{i} \rangle = 0 \, , \quad 1 \leq i \leq k \, . \\
&\langle X(v_{1}, \dotsc, v_{k}), X(v_{1}, \dotsc, v_{k}) \rangle = \det (\langle v_{i}, v_{j} \rangle) \, .
\end{align*}
We emphasize that the second axiom implies any such $X$ being alternating.

In particular, in the case $V$ carries an orientation $\mathcal{O}$, we say that an $(n-1)$-fold vector cross product $X$ is \textit{positively oriented} if the following condition holds for all $(n-1)$-tuples of linearly independent vectors $v_{1}, \dotsc, v_{n-1}$:
\begin{equation*}
\left(v_{1}, \dotsc, v_{n-1}, X(v_{1}, \dotsc, v_{n-1})\right) \in \mathcal{O} \, .
\end{equation*}
Analogously, a \textit{negatively oriented} vector cross product satisfies the same relation with $-\mathcal{O}$ in place of $\mathcal{O}$.

In \cite{brown1967}, Brown and Gray proved the following theorem:
\begin{theorem} \label{VCP-TH1}
Let $V$ be an oriented finite dimensional inner product space, of dimension $n$. There exists a unique positively oriented $(n-1)$-fold vector cross product $X = \cdot \times \dotsb \times \cdot$ on $V$. It is given by:
\begin{equation*}
v_{1} \times \dotsb \times v_{n-1} = \star (v_{1} \wedge \dotsb \wedge v_{n-1}) \, 
\end{equation*}
where $\star$ is the Hodge star operator on $V$.
\end{theorem}

We now turn our attention to manifolds. If $M$ is a smooth Riemannian manifold of dimension $m$, let $L^{k}\mathit{TM}$ be the disjoint union of all the vector spaces $L^{k}(T_{p}M)$:
$$ L^{k}\mathit{TM} = \bigsqcup_{p \in M} L^{k}(T_{p}M) \, .$$

Clearly, for $L^{k}(T_{p}M) \cong T^{(1,k)}(T_{p}M)$, the set $L^{k}\mathit{TM}$ has a canonical choice of topology and smooth structure turning it into a smooth vector bundle of rank $m^{k+1}$ over $M$. We define a $k$-\textit{fold vector cross product on} $M$, where $1 \leq k \leq m$, to be a smooth section $X$ of $L^{k}\mathit{TM}$ such that, for every point $p \in M$, the map $X_{p}$ is a $k$-fold vector cross product on $T_{p}M$.

We thus have the following corollary of Theorem \ref{VCP-TH1}:
\begin{corollary}
Let $M$ be a smooth oriented $m$-dimensional Riemannian manifold. There exists a unique $(m-1)$-fold positively oriented vector cross product on $M$. It acts on $(m-1)$-tuples of vector fields $X_{1}, \dotsc, X_{m-1}$ on $M$ by
\begin{equation*}
X_{1} \times \dotsb \times X_{m-1} = \star (X_{1} \wedge \dotsb \wedge X_{m-1}) \,. 
\end{equation*}
\end{corollary}

\section{Frames Along Curves}
\noindent In this section we review some basic facts about Euclidean submanifolds and orthonormal frames along curves.

Let us start with some notation. If $m\geq2$, let $M$ be an $m$-dimensional embedded submanifold of $\mathbb{R}^{d}$, and $\gamma \colon I=[0,\alpha] \to M$ a smooth regular curve in $M$. Throughout this paper, $\mathbb{R}^{d}$ will always be equipped with the standard Euclidean metric $\overline{g}$, typically indicated by a dot ``$\,\cdot\,$'', and standard orientation. Thus, there is a natural choice of Riemannian metric on $M$: the induced metric $\iota^{\ast}\overline{g}$, i.e., the pullback of $\overline{g}$ by the inclusion $\iota \colon M \hookrightarrow \mathbb{R}^{d}$.

Working with submanifolds, it is customary to identify each tangent space $T_{p}M$ with its image under the differential of $\iota$. In so doing, the ambient tangent space $T_{p}\mathbb{R}^{d}$ splits as the orthogonal direct sum $T_{p}M \oplus N_{p}M$, where $N_{p}M$ is the normal space of $M$ at $p$. Thus, the set $\mathfrak{X}(M)$ of tangent vector fields \emph{on} $M$ becomes a proper subset of the set of vector fields \emph{along} $M$, which we denote by $\overline{\mathfrak{X}}(M)$. If $X \in \mathfrak{X}(M)$ and $\varUpsilon \in \overline{\mathfrak{X}}(M)$,
\begin{equation*}
\overline{\nabla}_{X}\varUpsilon = (\overline{\nabla}_{X}\varUpsilon)^{\top} + (\overline{\nabla}_{X}\varUpsilon)^{\perp}\,,
\end{equation*}
where $\overline{\nabla}$ is the Euclidean connection, $\top$ and $\perp$ are the orthogonal projections onto the tangent and normal bundle of $M$, and where the vector fields $X$ and $\varUpsilon$ are extended arbitrarily to $\mathbb{R}^{d}$. It turns out that the map $\mathfrak{X}(M) \times \mathfrak{X}(M) \to  \mathfrak{X}(M)$ defined by
\begin{equation*}
(X,Y) \mapsto (\overline{\nabla}_{X}Y)^{\top}
\end{equation*}
is a linear connection on $M$, called the tangential connection. In fact, it is no other than the (intrinsic) Levi-Civita connection $\nabla$ of $(M,\iota^{\ast}\overline{g})$.

Similarly, indicating by $\mathfrak{X}(M)^{\perp}$ the set of normal vector fields along $M$, we define the normal connection on $M$ as the map $\mathfrak{X}(M) \times \mathfrak{X}(M)^{\perp} \to  \mathfrak{X}(M)^{\perp}$ given by
\begin{equation*}
(X,N) \mapsto (\overline{\nabla}_{X}N)^{\perp}\,.
\end{equation*}

Let us recall that an orthonormal frame along $\gamma$ is an $m$-tuple of smooth vector fields $(E_{i})_{i\,=\,1}^{m}$ along $\gamma$ such that $(E_{i}(t))_{i\,=\,1}^{m}$ is an orthonormal basis of $T_{\gamma(t)}M$ for all $t$. In particular, an orthonormal frame $(W_{1}, \dotsc, W_{d})$ along a curve $\iota \circ \gamma$ in $\mathbb{R}^{d}$ is said to be $M$-\textit{adapted} if $(W_{i})_{i\,=\,1}^{m}$ spans the ambient tangent bundle over $\gamma$.

In the remainder of this section, we assume that $M$ has codimension one in $\mathbb{R}^{d}$, i.e., that $d=m+1$. Under such hypothesis, given any orthonormal frame $(E_{i})_{i\,=\,1}^{m}$ along $\gamma$, we can construct an associated $M$-adapted orthonormal frame along $\iota \circ \gamma$ as follows. For $k=1, \dotsc, m$, let $W_{k} = E_{k}$; then, for $k=m+1$,
\begin{equation*}
W_{m+1} = E_{1} \times \dotsb \times E_{m}\,,
\end{equation*}
so that $(W_{1}, \dotsc, W_{m+1})$ is the unique extension of $(E_{i}(t))_{i\,=\,1}^{m}$ to a positively oriented, orthonormal frame along $\iota\circ\gamma$.

Denoting by $D_{t}$ and $\overline{D}_{t}$ the covariant derivative operators determined by $\nabla$ and $\overline{\nabla}$, respectively, we may write
\begin{equation} \label{FAC-EQ1}
\overline{D}_{t} E_{i} = D_{t} E_{i} +\tau_{i}W_{m+1} \,,
\end{equation}
for some smooth function $\tau_{i} \colon I \to \mathbb{R}$. Clearly, should $M$ be orientable, $\tau_{i} = \pm h(E_{1},E_{i})$, where $h$ is the (scalar) second fundamental form of $M$ determined by a choice of unit normal vector field. Moreover, it easily follows from orthonormality that
\begin{equation*}
\overline{D}_{t} W_{m+1} = -\tau_{1}E_{1} - \dotsb -\tau_{m}E_{m}\,.
\end{equation*}

\section{Developable Surfaces}
\noindent The main purpose of this section is to generalize to higher dimensions the following well-known fact about ruled surfaces in $\mathbb{R}^{3}$ -- see for example \cite[p.~194]{docarmo1976}:
\begin{lemma} \label{DH-LM1}
Let $I$, $J$ be open intervals. Further, let $\gamma$ and $X$ be curves $I \to \mathbb{R}^{3}$ such that the map $\sigma \colon I \times J \to \mathbb{R}^{3}$ given by
\begin{equation*}
\sigma(t,u) = \gamma(t)+ u X(t)
\end{equation*}
is a smooth injective immersion. Then the Gauss curvature of $\sigma(I \times J)$ is zero precisely when $\gamma$ and $X$ satisfy $\dot{\gamma} \cdot \dot{X} \times X = 0$.
\end{lemma}

We shall begin with some definitions extending the classical notions of ruled and torse surface to arbitrary dimension, yet keeping the codimension fixed to $1$. If $m \geq 2$, let $H$ be a hypersurface in $\mathbb{R}^{m+1}$, as always smooth and embedded.
\begin{definition} \label{DH-DEF2}
We say that $H$ is a \textit{ruled} surface if
\begin{enumerate}
\item \label{cond1} $H$ is free of planar points, that is, there exists no point of $H$ where the second fundamental form vanishes;
\item there exists a \textit{ruled structure on} $H$, that is, a foliation of $H$ by open subsets of $(m-1)$-dimensional affine subspaces of $\mathbb{R}^{m+1}$, called \textit{rulings}. 
\end{enumerate}
In particular, a ruled surface $H$ is said to be a \textit{torse surface} if, for every pair of points $(p,q)$ on the same ruling, we have $T_{p}H = T_{q}H$, i.e., if all tangent spaces of $H$ along a fixed ruling can be canonically identified with the same linear subspace of $\mathbb{R}^{m+1}$.
\end{definition}

\begin{remark} \label{DS-RMK3}
Although condition \ref{cond1} in Definition \ref{DH-DEF2} may seem overly restrictive, it gives any ruled surface $H$ a desirable property. Namely, it ensures the existence of a \emph{smooth} ruled parametrisation of $H$ \cite{ushakov1996}. On the other hand, we will also need to work with the broader class of \textit{generalised ruled hypersurfaces} obtained by relaxing such condition. It is well known that every generalised torse with planar points is made up of both standard torses and pieces of $m$-planes, always glued along a well-defined ruling.
\end{remark}

Remember that any $d$-dimensional Riemannian manifold locally isometric to $\mathbb{R}^{d}$ is said to be \textit{flat}. In particular, the classical term for hypersurfaces is \emph{developable}, see \cite[Section~1]{ushakov1999} for a detailed discussion on terminology. Remarkably, it turns out that

\begin{theorem}[{\cite[Theorem~1]{ushakov1999}}]
$H$ is a torse surface if and only if it is free of planar points and, when equipped with the induced metric $\iota^{\ast}\overline{g}$, $H$ becomes a flat Riemannian manifold.
\end{theorem}

\begin{corollary} \label{DS-COR}
$H$ is a generalised torse surface if and only if the induced metric on $H$ is flat.
\end{corollary}

Given a curve $\gamma$ in $\mathbb{R}^{m+1}$, the following result is key for constructing ruled surfaces containing $\gamma$. Note that in its statement we use the canonical isomorphism between $\mathbb{R}^{m+1}$ and any of its tangent spaces to identify the vector fields $X_{1}, \dotsc, X_{m-1}$ along $\gamma$ with curves in $\mathbb{R}^{m+1}$.

\begin{lemma}
Let $I$ be a closed interval. Let $\gamma \colon I \to \mathbb{R}^{m+1}$ be a smooth injective immersion. Let $(X_{1}, \dotsc, X_{m-1})$ be a smooth, linearly independent $(m-1)$-tuple of vector fields along $\gamma$ such that $\dot{\gamma}(t) \times X_{1}(t) \times \dotsb \times X_{m-1}(t) \neq 0$ for all $t \in I$. Then there exists an open box $V$ in $\mathbb{R}^{m-1}$ containing the origin such that the restriction to $I \times V$ of the map $\sigma \, \colon I \times \mathbb{R}^{m-1} \to \mathbb{R}^{m+1}$ defined by
\begin{equation*}
\sigma(t, u) = \gamma(t) + \sum\nolimits_{j} u^{j} X_{j}(t)
 \end{equation*}
is a smooth embedding.
\end{lemma}
\begin{proof}
To show that $\sigma$ restricts to an embedding, we first prove the existence of an open box $V_{1}$ such that $\sigma \rvert_{I\times V_{1}}$ is a smooth immersion. Essentially, the statement will then follow by compactness of $I$.

Obviously, $\sigma$ is immersive at $(t,u)$ if and only if the length $\ell \,\colon I \times \mathbb{R}^{m-1} \to \mathbb{R}$ of the cross product of the partial derivatives of $\sigma$ is non-zero at $(t,u)$. Thus, define $W_{t}$ to be the subset of $\{t\} \times \mathbb{R}^{m-1}$ where $\sigma$ is immersive. It is an open subset in $\mathbb{R}^{m-1}$ because it is the inverse image of an open set under a continuous map, $W_{t} = \ell(t,\cdot)^{-1}(\mathbb{R} \setminus \{0\})$; it contains $0$ by assumption. Thence, there exists an $\epsilon_{t} >0$ such that the open ball $B(\epsilon_{t},0) \subset \mathbb{R}^{m-1}$ is completely contained in $W_{t}$. Letting $\epsilon_{1} = \inf_{t \, \in \, I}(\epsilon_{t})$, we can conclude that the restriction of $\sigma$ to the box $I\times (-\epsilon_{1}/2,\epsilon_{1}/2)^{m-1}$ is a smooth immersion.

Now, being $\sigma$ a smooth immersion on $I \times V_{1}$, it follows that every point of $I \times V_{1}$ has a neighbourhood on which $\sigma$ is a smooth embedding. Let then $W_{t}'$ be the subset of $W_{t}$ where $\sigma$ is an embedding. It is open in $\mathbb{R}^{m-1}$, and it contains the origin because $\gamma$ is a smooth injective immersion of a compact manifold. From here we may proceed as before.
\end{proof}

Thus, for suitably chosen $(X_{1}, \dotsc, X_{m-1})$ and $V \subset \mathbb{R}^{m-1}$, we have verified that $H_{\sigma} = \Ima\sigma\rvert_{I\times V}$ is a hypersurface in $\mathbb{R}^{m+1}$, and $\mathscr{F}_{\sigma} = \{ \sigma(t,V) \}_{t \, \in \, I}$ a ruled structure on it. Under such hypothesis, let us assume $H_{\sigma}$ is orientable (this we can do, possibly limiting the analysis to an open subset). Then, we may pick out a smooth unit normal vector field $N$ along $H_{\sigma}$ by means of the $m$-fold cross product on $\mathbb{R}^{m+1}$, as follows. Letting
\begin{equation} \label{DH-EQ1}
Z = \frac{\partial \sigma}{\partial t} \times \frac{\partial \sigma}{\partial u^{1}} \times \dotsb \times \frac{\partial \sigma}{\partial u^{m-1}} \, ,
\end{equation}
define $\widehat{N} = Z \lvert Z \rvert^{-1}$, and so $N = \widehat{N} \circ \sigma^{-1}$. In this situation, assuming there are no planar points, $H_{\sigma}$ being a torse surface is equivalent to $N$ being constant along each of the rulings. Thus, indicating with $\overline{\nabla}$ the Euclidean connection on $\mathbb{R}^{m+1}$, $(H_{\sigma},\iota^{\ast}\overline{g})$ is flat if and only if, for all vector fields $X$ tangent to $\mathscr{F}_{\sigma}$ on $H_{\sigma}$: 
\begin{equation} \label{DH-EQ2}
 \overline{\nabla}_{X}N = 0\, .
\end{equation}
In fact, by linearity -- and writing $\partial_{j}$ as a shorthand for $\frac{\partial}{\partial u^{j}}$ -- it suffices that \eqref{DH-EQ2} holds for the vector fields $\sigma_{\ast}(\partial_{1}), \dotsc, \sigma_{\ast}(\partial_{m-1})$ spanning the distribution corresponding to $\mathscr{F}_{\sigma}$. We may thereby express the developability condition for $(H_{\sigma},\iota^{\ast}\overline{g})$ simply as
\begin{equation} \label{DH-EQ3}
\partial_{1}\widehat{N} = \dotsb = \partial_{m-1}\widehat{N} = 0\,,
\end{equation}
where we understand $\partial_{j}$ as acting on the coordinate functions $\widehat{N}^{1},\dotsc, \widehat{N}^{m+1}$ of $\widehat{N}$ in the standard coordinate frame of $T \mathbb{R}^{m+1}$.

The next lemma finally translates \eqref{DH-EQ3} into $m-1$ conditions involving the vector fields $X_{1}, \dotsc, X_{m-1}$ along $\gamma$, and represents the sought generalization of Lemma \ref{DH-LM1}. It says that $\iota^{\ast}\overline{g}$ is a flat Riemannian metric precisely when $\overline{D}_{t}X_{j} = D_{t}X_{j}$ for every $j$, or equivalently when each of the normal projections $(\overline{D}_{t}X_{1})^{\perp}, \dotsc, (\overline{D}_{t}X_{m-1})^{\perp}$ vanishes identically.
\begin{lemma}
Assume $\sigma\rvert_{I\times V}$ is a smooth embedding. The hypersurface $H_{\sigma}$ is a generalised torse surface if and only if the following equations hold:
\begin{align} \label{DH-EQ4}
\begin{split}
\dot{\gamma} \cdot \partial_{1} Z \equiv \dot{\gamma} \cdot \overline{D}_{t}X_{1} \times X_{1} \times \dotsb \times X_{m-1} &= 0 \\
&\mathrel{\makebox[\widthof{=}]{\vdots}} \\
\dot{\gamma} \cdot \partial_{m-1} Z \equiv \dot{\gamma} \cdot \overline{D}_{t}X_{m-1} \times X_{1} \times \dotsb \times X_{m-1} &= 0
\end{split}
\end{align}
\end{lemma}
\begin{proof}
Computing the partial derivatives of $\sigma$ and substituting them into the expression \eqref{DH-EQ1} for $Z$, we get:
\begin{equation*}
Z(t,u)= \{ \dot{\gamma}(t) + u^{i} \overline{D}_{t}X_{i}(t) \} \times X_{1}(t) \times \dotsb \times X_{m-1}(t)\,,
\end{equation*}
from which the identity $\partial_{j}Z \equiv \overline{D}_{t}X_{j} \times X_{1} \times \dotsb \times X_{m-1}$ clearly follows. Thus, we need to prove that $\partial_{1}\widehat{N} = \dotsb = \partial_{m-1}\widehat{N} = 0$ if and only if $\partial_{1}Z \cdot \dot{\gamma} = \dotsb = \partial_{m-1}Z \cdot \dot{\gamma} = 0$. In fact, for $\partial_{j}Z$ is orthogonal to $X_{1}, \dotsc, X_{m-1}$, it is enough to check that $\partial_{1}\widehat{N} = \dotsb = \partial_{m-1}\widehat{N} = 0$ if and only if $(\partial_{1}Z )^{\top}= \dotsb = (\partial_{m-1}Z)^{\top}= 0$. First, assume $\partial_{j}\widehat{N} = 0$. Since $\widehat{N} = Z \lvert Z \rvert^{-1}$, it follows by linearity of the tangential projection that
\begin{equation*}
\lvert Z \rvert (\partial_{j}Z)^{\top} - Z^{\top} \partial_{j} \lvert Z \rvert = 0\,,
\end{equation*}
which is true exactly when $(\partial_{j}Z)^{\top} = 0$, as desired. To verify the converse, note that $(\partial_{j}N)^{\perp} = 0$ because $N$ has unit length. Thus, again by linearity of $\top$,
\begin{equation*}
\partial_{j}\widehat{N}= \frac {(\partial_{j}Z)^{\top}\lvert Z\rvert - Z^{\top}\partial_{j}\lvert Z \rvert}{\lvert Z\rvert^{2}}\,.
\end{equation*}
Since $Z^{\top}=0$, the claim follows.
\end{proof}

\section{Proof of the Main Result}

\noindent Here we prove our main result, stated in Theorem \ref{mainresult} in the Introduction. The proof is constructive and is based on the fact that an Euclidean hypersurface without planar points has a flat induced metric precisely when it is a torse surface (Theorem 1.3). Let $M$ be a hypersurface in $\mathbb{R}^{m+1}$ and $\gamma$ a smooth curve in $M$, as defined at the beginning of Section $3$. Denoting by $\mathfrak{X}(\gamma)$ the set of smooth, non-vanishing vector fields along $\gamma$, define an equivalence relation on the $n$-th Cartesian power $\mathfrak{X}(\gamma)^{n}$ of $\mathfrak{X}(\gamma)$ by the following rule: 
\begin{equation*}
\{(X_{1},\dotsc,X_{n}) \sim (Y_{1},\dotsc,Y_{n})\} \Leftrightarrow \{\Span (X_{1},\dotsc,X_{n}) =\Span (Y_{1},\dotsc,Y_{n})\}.
\end{equation*}
Let us indicate an element of the quotient $\mathfrak{X}(\gamma)^{n}/{\sim}$, that is, an element of $\mathfrak{X}(\gamma)^{n}$ up to equivalence, by $[X_{1},\dotsc,X_{n}]$. We wish to find $[X_{1}, \dotsc, X_{m-1}]$ such that, for every $t \in I$ and integer $j$ with $1 \leq j \leq m-1$, both the conditions
\begin{align}
&\dot{\gamma} \cdot \overline{D}_{t}X_{j} \times X_{1} \times \dotsb \times X_{m-1} = 0 \label{PMR-EQ1} \\
&\dot{\gamma}(t) \times X_{1}(t) \times \dotsb \times X_{m-1}(t) \neq 0 \label{PMR-EQ2}
\end{align}
are satisfied. Beware that, throughout this section, we will extensively use Einstein summation convention: every time the same index appears twice in any monomial expression, once as an upper index and once as a lower index, summation over all possible values of that index is understood. 

Once and for all, let us choose a $\gamma$\textit{-adapted} orthonormal frame $(E_{1}, \dotsc, E_{m})$ along $\gamma$: this is just an orthonormal frame along $\gamma$ whose first element coincides with the tangent vector $\dot{\gamma}$. The first step is to rewrite \eqref{PMR-EQ1} as an equation involving the $m(m-1)$ coordinate functions $X_{j}^{i}$ of $X_{1},\dotsc,X_{m-1}$ with respect to $(E_{1}, \dotsc, E_{m})$. Differentiating covariantly $X_{j} = X_{j}^{i}E_{i}$ and substituting, we obtain
\begin{equation} \label{PMR-EQ3}
E_{1} \cdot (\overline{D}_{t}X_{j}^{i}E_{i} + X_{j}^{i} \overline{D}_{t}E_{i}) \times X_{1}^{i}E_{i} \times \dotsb \times X_{m-1}^{i}E_{i} = 0 \,,
\end{equation}
whereas, from \eqref{FAC-EQ1}.
\begin{align*}
\sum_{i\, =\, 1}^{m} \overline{D}_{t}E_{i}&=\sum_{i \, =\, 1}^{m} D_{t}E_{i}+ E_{m+1}\sum_{i \,=\, 1}^{m}\tau_{i} \\
&= \sum_{i \, =\, 1}^{m} \left\{ (D_{t}E_{i} \cdot E_{1}) E_{1} + \dotsb + (D_{t}E_{i} \cdot E_{m}) E_{m}\right\} + E_{m+1}\sum_{i \,=\, 1}^{m}\tau_{i} \,.
\end{align*}
Now, given any ordered $m$-tuple $(i_{1},\dotsc, i_{m})$ of integers with $1 \leq i_{1} \leq m+1$ and $1 \leq i_{k} \leq m$ for $k = 2,\dotsc,m$, a necessary condition for the $m$-fold cross product $E_{i_{1}}\times \dotsb \times E_{i_{m}}$ to give either $E_{1}$ or $-E_{1}$ is that $i_{1} = m+1$ and $i_{k} \neq 1$. It follows that \eqref{PMR-EQ3} is equivalent to
\begin{equation} \label{PMR-EQ4}
E_{1} \cdot X_{j}^{i}\tau_{i}E_{m+1} \times (X_{1}^{2}E_{2} + \dotsb + X_{1}^{m}E_{m}) \times \dotsb \times (X_{m-1}^{2}E_{2} + \dotsb + X_{m-1}^{m}E_{m}) = 0\,.
\end{equation}
In fact, $E_{i_{1}}\times \dotsb \times E_{i_{m}} = \pm E_{1}$ if and only if $i_{1} = m+1$ and the $(m-1)$-tuple $(i_{2},\dotsc, i_{m})$ is a permutation of $(2,\dotsc,m)$. In particular, if it is an \emph{even} permutation, then the basis $(E_{m+1},E_{i_{2}},\dotsc,E_{i_{m}},E_{1})$ is \emph{negatively} oriented, for transposing $E_{m+1}$ and $E_{1}$ must give a positive basis, and so $E_{i_{1}}\times \dotsb \times E_{i_{m}} = -E_{1}$. Thence, denoting by $S_{m}^{2}$ the group of permutations $\sigma$ of $(2,\dotsc,m)$, we may write \eqref{PMR-EQ4} simply as
\begin{equation*}
-X_{j}^{i}\tau_{i} \sum_{\sigma \,\in\, S_{m}^{2}} \Sign(\sigma) X_{1}^{\sigma(2)} \dotsm X_{m-1}^{\sigma(m)} = 0\,.
\end{equation*}
On the other hand, a similar computation would reveal that condition \eqref{PMR-EQ2} is satisfied for every $t$ if and only if the summation term above (the term independent of $j$) never vanishes. We may thereby conclude that, under the assumption of \eqref{PMR-EQ2} being true, condition \eqref{PMR-EQ1} is equivalent to $X_{j}^{i}\tau_{i} = 0$.

Next, consider the set $\mathscr{Z} \subset \mathfrak{X}(\gamma)$ of smooth vector fields $Z$ along $\gamma$ such that $Z^{1}(t) = Z \cdot E_{1}(t) \neq0$ for every $t$. We establish a bijection between its quotient $\mathscr{Z}/{\sim}$ by $\sim$ and the subset of $\mathfrak{X}(\gamma)^{m-1}/{\sim}$ where \eqref{PMR-EQ2} holds. For every $j$, let
\begin{equation} \label{PMR-EQ5}
X_{j}(Z) = Z \times E_{2} \times \dotsb \times \widetilde{E}_{m-j+1} \times \dotsb \times E_{m}\,,
\end{equation}
where the tilde indicates that $E_{m-j+1}$ is omitted, so that the cross product is $(m-1)$-fold. For example, when $j=1$, we omit the last vector field $E_{m}$; when $j=2$ the second to last, and so on, until dropping $E_{2}$ for $j=m-1$. Linear independence of $E_{1},X_{1}(Z), \dotsc, X_{m-1}(Z)$ is easily seen, as by definition $Z$ is never in the span of $E_{2},\dotsc,E_{m}$. Since the normal projection $Z \mapsto Z^{\perp}$ induces a bijection between $\mathscr{Z}/{\sim}$ and the set of smooth $(m-1)$-distributions along $\gamma$ nowhere parallel to $E_{1}$, it follows that the map $[Z] \mapsto [X_{1}(Z),\dotsc,X_{m-1}(Z)]$ between classes of equivalence is indeed a valid parametrisation of the solution set of \eqref{PMR-EQ2}.

We then compute the coordinates of the cross product in \eqref{PMR-EQ5} with respect to the frame $(E_{1},\dotsc,E_{m})$. Substituting $Z = Z^{i}E_{i}$, all but the terms $Z^{1}E_{1}$ and $Z^{m-j+1}E_{m-j+1}$ will not give any contribution. In particular, $E_{1} \times \dotsb \times \widetilde{E}_{m-j+1} \times \dotsb \times E_{m} = \pm E_{m-j+1}$ depending on whether $(E_{1}, \dotsc, \widetilde{E}_{m-j+1}, \dotsc, E_{m}, E_{m-j+1})$ is positively or negatively oriented. Since the corresponding permutation of $(1,\dotsc, m)$ has sign $(-1)^{j-1}$, we conclude that $X_{j}^{m-j+1}(Z) = (-1)^{j-1}Z^{1}$. An analogous argument would show that $X_{j}^{1}(Z) = (-1)^{j}Z^{m-j+1}$.

Summing up, solving the original problem on $\mathfrak{X}(\gamma)^{m-1}/{\sim}$ essentially amounts to finding $[Z] \in \mathscr{Z}/{\sim}$ such that $X_{j}^{i}(Z)\tau_{i} = 0$ for every $j$. Moreover, by the previous computation,
\begin{equation*}
X_{j}^{i}(Z)\tau_{i} = (-1)^{j}Z^{m-j+1}\tau_{1}+(-1)^{j-1}Z^{1}\tau_{m-j+1}\,.
\end{equation*}

Thus, denoting again by $\sim$ the equivalence relation on $C^{\infty}(I)^{m} = C^{\infty}(I;\mathbb{R}^{m})$ naturally induced from the one on $\mathfrak{X}(\gamma)$, we need to look for $(Z^{1}, \dotsc,Z^{m})$, up to equivalence, satisfying the following system of $m-1$ linear equations on $C^{\infty}(I;\mathbb{R}_{\neq0}) \times C^{\infty}(I)^{m-1}$:
\begin{align}
\begin{split} \label{PMR-EQ6}
Z^{m}\tau_{1}-Z^{1}\tau_{m} &=0\\
Z^{m-1}\tau_{1}-Z^{1}\tau_{m-1} &=0\\
&\mathrel{\makebox[\widthof{=}]{\vdots}} \\
Z^{3}\tau_{1}-Z^{1}\tau_{3} &=0\\
Z^{2}\tau_{1}-Z^{1}\tau_{2} &=0\,.
\end{split}
\end{align}
Assume $\tau_{1}(t) \neq0$ for all $t$. Then, for any given $Z^{1}$ (remember $Z^{1}$ is non-vanishing by definition), the system has solution
\begin{equation*}
\frac{Z^{1}}{\tau_{1}}\left(\tau_{1},\dotsc,\tau_{m}\right).
\end{equation*}
However, it is easy to see that all solutions are in one and the same equivalence class. Indeed, if $f$ and $g$ are two distinct values of $Z^{1}$, then
\begin{equation*}
\frac{\tau_{i}}{\tau_{1}}f = \frac{f}{g}\frac{\tau_{i}}{\tau_{1}}g\,.
\end{equation*}

In particular, letting $Z^{1}=\tau_{1}$, we obtain $Z^{i} = \tau_{i}$ for every $i=1,\dotsc,m$, and the solution of the original problem on $\mathfrak{X}(\gamma)^{m-1}/{\sim}$ is given by
\begin{align*}
X_{1} &= -\tau_{m}E_{1} + \tau_{1}E_{m}\\
X_{2} &= \tau_{m-1}E_{1} - \tau_{1}E_{m-1}\\
&\mathrel{\makebox[\widthof{=}]{\vdots}} \\
X_{m-2} &= (-1)^{m-2} \tau_{3}E_{1} + (-1)^{m-3}\tau_{1}E_{3}\\
X_{m-1} &= (-1)^{m-1} \tau_{2}E_{1} + (-1)^{m-2}\tau_{1}E_{2}\,.
\end{align*}

As for uniqueness, in view of Remark \ref{DS-RMK3}, it is sufficient to show that the condition $\tau_{1}(t) \neq 0$ for all $t$ implies any flat approximation $H$ of $M^{m}$ along $\gamma$ be free of planar points, i.e., be a torse surface. To see this, let $N_{H}$ and $N_{M}$ be smooth unit normal vector fields along $H$ and $M^{m}$, respectively, defined in a neighbourhood of $\gamma(t)$. Then, $\overline{D}_{t}N _{H} =\overline{D}_{t}N _{M}$. Since $H$ is a generalised torse surface by Corollary \ref{DS-COR}, the claim easily follows.
%
%

\section{Construction of an Adapted Frame}
\noindent As seen in the last section, the construction of the flat approximation of $M$ along $\gamma$ requires choosing some $\gamma$-adapted orthonormal frame $(E_{i})_{i \,= \,1}^{m}$ along $\gamma$. We emphasize that such a choice is completely arbitrary. If the curve in question satisfies some (rather strong) conditions on its derivatives, then a natural generalization of the classical Frenet--Serret frame is available. The reader may find details on this construction in \cite{spivak1999} or \cite{kuhnel2015}. Here we briefly review an alternative approach, one that does not require any initial assumption on the curve. Such approach is due to Bishop \cite{bishop1975}.

First of all, since the problem is local, we are free to assume that $\gamma$ is a smooth embedding. Thus, for any point $p \in S = \gamma(I)$, there exist slice coordinates $(x_{1},\dotsc,x_{m})$ in a neighbourhood $U$ of $p$. It follows that $(\partial_{1}\lvert_{p}, \dotsc, \partial_{m}\lvert_{p})$ is a $\gamma$-adapted basis of $T_{p}M$, i.e., it satisfies $T_{p}S = \Span \partial_{1}\lvert_{p}$ and $N_{p}S = \Span(\partial_{2}\lvert_{p}, \dotsc, \partial_{m}\lvert_{p})$. By applying the Gram--Schmidt process to these vectors, one obtains an orthonormal basis $(n_{j})$ of $N_{p}S$. Although this basis is by no means canonical, the normal connection $\nabla^{\perp}$ of $S$ provides an obvious means for extending it to a frame for the normal bundle of $S$: for each $j$, let $\varUpsilon_{j}$ be the unique normal parallel vector field along $\gamma$ such that $\varUpsilon_{j}\lvert_{p} = n_{j}$ -- see \cite[p.~119]{oneill1983}. Because normal parallel translation is an isometry, the frame $(\dot{\gamma}, \varUpsilon_{1}, \dotsc, \varUpsilon_{m-1})$ is an orthonormal adapted frame along $\gamma$, as desired.

\bibliographystyle{unsrt}
\bibliography{NewText}
\end{document}